\documentclass[a4paper,11pt]{article}

\usepackage{amsmath,amsfonts,amssymb,amsthm}
\usepackage{tikz,graphicx}
\usepackage[absolute,overlay]{textpos} 
\usetikzlibrary{decorations.pathmorphing,matrix,decorations.pathreplacing,arrows}
 \usepackage{url}
\usepackage{breqn}
\usepackage[lined,algonl,boxed,ruled,vlined,linesnumbered,resetcount]{algorithm2e}
\usepackage{pifont}
\usepackage{multirow}

 \newtheorem{theorem}{Theorem}[section]

 \newtheorem{proposition}[theorem]{Proposition}

\begin{document}

\title{On Irreducible Polynomials of the Form $b(x^d)$}

\author{
  Palash Sarkar\\
    \small Indian Statistical Institute\\[-0.8ex]
    \small \texttt{palash@isical.ac.in}\\
  \and
  Shashank Singh\\
    \small Indian Statistical Institute\\[-0.8ex]
    \small \texttt{sha2nk.singh@gmail.com}
}

\date{}

\maketitle

\begin{abstract}
Let $p$ be a prime and $b(x)$ be an irreducible polynomial of degree $k$ over $\mathbb{F}_p$. Let $d\geq 1$ be an integer.
Consider the following question: Is $b(x^d)$ irreducible? We derive necessary conditions for $b(x^d)$ to be irreducible. Further,
when the necessary conditions are satisfied, we obtain the probability for $b(x^d)$ to be irreducible.
\end{abstract}

\section{Introduction\label{sec-intro}}
Irreducible polynomials over finite fields are of practical interest in areas such as coding theory and cryptography. 
Suppose $b(x)$ is an irreducible polynomial over $\mathbb{F}_p$ where $p$ is a prime and $d\geq 1$ is a positive integer.
We investigate the question of when $b(x^d)$ is irreducible over $\mathbb{F}_p$. This simple question turns out to have a 
surprisingly interesting answer which is provided in the rest of the paper.

One possible practical interest in the question arises from the need of obtaining sparse irreducible polynomials. The ability 
to move from an irreducible $b(x)$ to an irreducible $b(x^d)$ allows a simple way to obtain irreducible polynomials of 
progressively higher degrees. The question may be further investigated in contexts where it is directly relevant.

\section{Conditions for Irreducibility\label{sec-cond}}
We are interested in the polynomial $b(x^d)$. If $d$ is a multiple of the characteristic of the field, then clearly
$b(x^d)$ is reducible. We, however, do not need to consider this condition separately. Later, we will show that one of the
conditions that we derive imply this condition.

The following result relates the irreducibility of a polynomial to the structure of its roots.
\begin{proposition}\label{prop-irred}
A polynomial $g(x)$ of degree $t$ is irreducible over the field $\mathbb{F}_q$ if and only if its $t$
distinct roots are $\beta,\beta^q,\ldots,\beta^{q^{t-1}}$ for some non-zero $\beta\in\mathbb{F}_{q^t}$.
\end{proposition}
\begin{proof}
If $g(x)$ is irreducible, then this is a known result. (See for example~\cite{LN94}.) The other side is also quite simple.
So, suppose the $t$ distinct roots of $g(x)$ in $\mathbb{F}_{q^t}$ are $\beta,\beta^q,\ldots,\beta^{q^{t-1}}$ for some 
$\beta\in\mathbb{F}_{q^t}$ with $\beta^{q^t}=\beta$. 
Let if possible, $a(x)$ be an irreducible factor of $g(x)$ over $\mathbb{F}_q$ with ${\rm deg}(a)=s<t$. Then the $s$ distinct roots of $a(x)$
in $\mathbb{F}_{q^s}$ are $\gamma,\gamma^q,\ldots,\gamma^{q^{s-1}}$ for some $\gamma\in\mathbb{F}_{q^s}$ with $\gamma^{q^s}=\gamma$. 
Since $a(x)$ is a factor of $g(x)$, we must have $\gamma=\beta^{q^i}$ for some $0\leq i<t$. Then the relation
$\gamma^{q^s}=\gamma$ leads to $\beta^{q^{s+i}}=\beta^{q^i}$ which shows that $s+i\equiv i\bmod t$ implying that $t|s$. Since $s<t$
this is not possible leading to a contradiction. So, $g(x)$ cannot have any irreducible factor $a(x)$ of degree less than $t$.
\end{proof}

For $b(x^d)$ to be irreducible, it is clearly necessary that $b(x)$ itself must be irreducible. We next characterise the condition under 
which $b(x^d)$ is also irreducible. 

\begin{proposition}\label{prop:red-criterion}
Let $\mathbb{F}_q$ be a field and suppose $b(x)$ is a polynomial of degree $m$ which is irreducible over $\mathbb{F}_{q}$. 
Let $\alpha$ be a root of $b(x)$ in $\mathbb{F}_{q^{m}}$. Then $b(x^d)$ is irreducible over $\mathbb{F}_{q}$ if and only if $x^d-\alpha$ 
is irreducible over $\mathbb{F}_{q^{m}}$.
\end{proposition}
\begin{proof}
We will require the basic result that a polynomial $g(x)$ of degree $t$ is irreducible over $\mathbb{F}_q$ if and only if its $t$
distinct roots are $\beta,\beta^q,\ldots,\beta^{q^{t-1}}$ for some non-zero $\beta\in\mathbb{F}_{q^t}$.

The $m$ roots of $b(x)$ in $\mathbb{F}_{q^{m}}$ are $\alpha, \alpha^q, \ldots,\alpha^{q^{m-1}}$. So over $\mathbb{F}_{q^m}$,
 \begin{eqnarray*}
  b(x)   &=& (x-\alpha)\times(x-\alpha^q)\times \ldots \times (x-\alpha^{q^{m-1}}); \\
  b(x^d) &=& (x^d-\alpha)\times(x^d-\alpha^q)\times \ldots \times (x^d-\alpha^{q^{m-1}}).
 \end{eqnarray*}

Now suppose that $(x^d-\alpha)$ is irreducible over $\mathbb{F}_{q^{m}}$. Let $\beta \in \mathbb{F}_{q^{md}}$ be a root of $(x^d-\alpha)$
and hence $\beta,\beta^{q^{m}},\ldots,\beta^{q^{(d-1)m}}$ are all the $d$ distinct roots of $(x^d-\alpha)$ in $\mathbb{F}_{q^{md}}$
and $\beta^{q^{md}}=\beta$.

Note that if $\gamma$ is a root of $(x^d-\alpha) \in \mathbb{F}_{q^{m}}[x]$, then $\gamma^{q^i}$ is a root of   $(x^d-\alpha^{q^i})$.
Hence the elements of the multiset 
\begin{eqnarray*}
S &=& \lbrace \beta^{ q^{ (jm+i) } } \mbox{ where } 0 \leq i \leq m-1 \mbox{ and }  0 \leq j \leq d-1 \rbrace \\
  &=& \lbrace \beta^{q^k} \mbox{ where } 0 \leq k \leq md-1 \rbrace
\end{eqnarray*}
are all roots of $b(x^d)$. The elements of $S$ are distinct, since for $0\leq k_1<k_2<md$, $\beta^{q^{k_1}}=\beta^{q^{k_2}}$ implies 
that $k_1\equiv k_2\bmod md$ which is not possible
for the given range of $k_1$ and $k_2$. Thus, $b(x^d) \in \mathbb{F}_{q}[x]$, with degree $md$, has $md$ distinct roots in 
$\mathbb{F}_{q^{md}}$ given by the elements of $S$. Using Proposition~\ref{prop-irred}, $b(x^d)$ is irreducible over 
$\mathbb{F}_{q}[x]$.

Conversely, suppose that $(x^d-\alpha)$ is reducible in $\mathbb{F}_{q^{m}}[x]$ and let $\tau(x)$ be a nontrivial irreducible factor of 
degree $d_1<d$.
The splitting field of $\tau(x)$ is $\mathbb{F}_{q^{md_1}}$. Let $\delta \in \mathbb{F}_{q^{md_1}}$ be a root of $\tau(x)$. All the distinct roots of $\tau(x)$
are $\delta, \delta^{q^m},\ldots, \delta^{q^{md_1-1}}$. They are also the roots of $(x^d-\alpha)$. Using the arguments similar 
to the proof of first part of the theorem, the elements of the set 
\begin{eqnarray*}
S_1 &=& \lbrace \delta^{ q^{ (jm+i) } } \mbox{ where } 0 \leq i \leq m-1 \mbox{ and }  0 \leq j \leq d_1-1 \rbrace \\
  &=& \lbrace \delta^{q^k} \mbox{ where } 0 \leq k \leq md_1-1 \rbrace
\end{eqnarray*}
are the distinct roots of $b(x^d)$ in $\mathbb{F}_{q^{md_1}}$. Let 
\begin{eqnarray*}
\pi(x)= \prod_{k=0}^{md_1-1} (x-\delta^{q^k}).
\end{eqnarray*}
The coefficient of $\pi(x)$ are the symmetric functions of $\delta, \delta^{q^m},\ldots, \delta^{q^{md_1-1}}$ in $\mathbb{F}_{q^{md_1}}$ and $\delta^{md_1}=\delta$.
So, the operation of raising to the power $q$ leaves the set $\{\delta, \delta^{q^m},\ldots, \delta^{q^{md_1-1}}\}$ and hence the symmetric
functions invariant. So these symmetric function are in $\mathbb{F}_{q}$. This implies that $\pi(x) \in \mathbb{F}_q[x]$. Since $md_1 < md$, $\pi(x)$ is a 
proper factor of $b(x^d)$ and so $b(x^d)$ is reducible.
\end{proof}

Proposition~\ref{prop:red-criterion} reduces the problem of determining irreducibility of the polynomial $b(x^d)$ to that of
determining the irreducibility of $x^d-\alpha$ for some root $\alpha$ of $b(x)$. The following result characterises the condition under 
which $x^d-\alpha$ will be irreducible~\cite{Va1895,Ca1901,Al02}.

\begin{proposition}[Vahlen-Capelli criterion]\label{prop:vahlen-capelli}
Let $F$ be an arbitrary field. Let $d$ be a positive integer and $\alpha\in F$. Then $x^d -\alpha$ is reducible in $F[x]$ if and only if 
one of the following two conditions hold.
\begin{enumerate}
\item There is a $\beta\in F$ such that $\alpha=\beta^{d^{\prime}}$ for some prime divisor $d^{\prime}$ of $d$; or,
\item there is a $\gamma$ in $F^{\star}$ such that $-4\alpha=\gamma^4$ whenever $4|d$.
\end{enumerate}
\end{proposition}
The following simple condition for reducibility can be proved without using the Vahlen-Capelli criterion.
\begin{proposition}\label{prop-simple-irred}
Let $\mathbb{F}_{q^k}$ be a finite field and $d\geq 1$ be an integer such that there is a prime divisor $d^{\prime}$ of $d$ for
which $\gcd(q^k-1,d^{\prime})=1$. Then for every $\alpha\in\mathbb{F}_{q^k}^{\star}$ the polynomial $x^d-\alpha$ is reducible
over $\mathbb{F}_{q^k}$.
\end{proposition}
\begin{proof}
Let $\beta$ be a generator of $\mathbb{F}_{q^k}^{\star}$. Since $d^{\prime}$ is co-prime to $q^k-1$, it follows that $\beta^{d^{\prime}}$ 
is also a generator of $\mathbb{F}_{q^k}^{\star}$. So, there is an $i$ such that $\alpha=(\beta^{d^{\prime}})^i$. Then we have 
$\alpha=\gamma^{d^{\prime}}$ where $\gamma=\beta^i$. In this case, the polynomial $(x^{d/d^{\prime}}-\gamma)$ is a factor
of $x^d-\alpha$.
\end{proof}
Let the characteristic of $\mathbb{F}_{q^k}$ be $p$. If $p|d$, then taking $d^{\prime}=p$, we have $\gcd(p,q^k-1)=1$ and
so $x^d-\alpha$ is reducible for every $\alpha\in\mathbb{F}_{q^k}^{\star}$. In particular, $b(x^d)$ is reducible. This, of course,
can be seen directly since $b(x^d)=(b(x^{\ell}))^p$ where $d=p\ell$.

We next prove a result on reducibility which follows from the Vahlen-Capelli criterion.
\begin{proposition}\label{prop:4-nec-cond}
Let $p$ be a prime congruent to $3$ modulo $4$, $d$ be a positive integer such that $4$ divides $d$ and $k$ be an odd positive
integer. Then for every $\alpha\in\mathbb{F}_{p^k}$, the polynomial $x^d-\alpha$ is reducible over $\mathbb{F}_{p^k}$.
\end{proposition}
\begin{proof}
Let $F=\mathbb{F}_{p^{k}}$ and $\beta$ be a generator of $F^{\star}$. Let $q=p^{k}$. Since $4|d$,
from the Vahlen-Capelli criterion, $x^d-\alpha$ is reducible if either $\alpha=\gamma^{d^{\prime}}$ for some divisor 
$d^{\prime}$ of $d$ and some element $\gamma$ of $F$; or, $-4\alpha=\delta^4$ for some $\delta\in F^{\star}$.
We show that for odd $k$, either $\alpha=\gamma^2$ (taking $d^{\prime}=2$) or, $-4\alpha=\delta^4$. This is achieved
by showing that for odd $k$, if $\alpha$ is a quadratic non-residue in $F^{\star}$, then $-4\alpha=\delta^4$.

So, assume that $\alpha$ is quadratic non-residue in $F^{\star}$, i.e., $\alpha=\beta^{t_1}$, where $t_1$ is odd and less than $q-1$.

An element $\beta^i$ is in $\mathbb{F}_{p}$ if and only if $\beta^{ip}=\beta^i$, i.e., if and only if,
$(p^{k}-1)|i(p-1)$. Let $M=(p^{k}-1)/(p-1)$. Then $\beta^i$ is in $\mathbb{F}_{p}$ if and only if $i$ is a multiple of
$M$. Also, since $k$ is odd, $M=(1+p+p^2+\cdots+p^{k-1})$ is odd. 

Since $-4\in\mathbb{F}_p$, we can write $-4=\beta^{Mj}$ for some $j$ in $\{0,\ldots,p-1\}$. Note that $4$ is a quadratic
residue in $\mathbb{F}_p$ and since $p\equiv 3\bmod 4$, $-1$ is a quadratic non-residue in $\mathbb{F}_p$. So,
$-4$ is a quadratic non-residue in $\mathbb{F}_p$. Since $M$ is odd and $-4$ is a quadratic non-residue, we have that the $j$ 
in $-4=\beta^{Mj}$ must also necessarily be odd. So we can write  $-4=\beta^{t_2}$, where $t_2$ is odd and is less than $q-1$.
Now we have,
\begin{eqnarray*}
 -4\alpha=\beta^{(t_1+t_2)\bmod (q-1)}
\end{eqnarray*}
Since $t_1$ and $t_2$ are odd, $t_1+t_2$ is congruent to either $0$ or $2$ modulo $4$. We also have $t_1+t_2 < 2(q-1)$.
\newline
{\bf Case~1:} Suppose that $t_1+t_2$ is congruent to $0$ modulo $4$. Then $t_1+t_2=4t$ where $t<q-1$. Thus,
	$-4\alpha = \beta^{4t}=\delta^4$ where $\delta=\beta^t$.
\newline
{\bf Case~2:} Suppose $t_1+t_2$ is equal to $2$ modulo $4$. Since $p\equiv 3\bmod 4$ and $k$ is odd, $q-1\equiv 2\bmod 4$ and so
$t_1+t_2+q-1$ is equal to $0$ modulo $4$. Then we can write, $t_1+t_2+q-1=4t^\prime$ where $t^\prime <q-1$. Again, we
can write $-4\alpha=\delta^4$ where $\delta=\beta^{t^{\prime}}$.
\end{proof}

The above conditions are sufficient to ensure reducibility of $x^d-\alpha$ and so if $x^d-\alpha$ is irreducible, then the negation
of the stated conditions must hold. We next prove the sufficiency of these conditions.

\begin{proposition}\label{prop-suff}
Let $p$ be a prime, $d,k\geq 1$ be integers. Consider the following condition.
$$\left.
\mbox{
\begin{minipage}{300pt}
Every prime divisor $d^{\prime}$ of $d$ divides $p^k-1$; \\
and \\
either ($4$ does not divide $d$);  \\
or, ($4|d$ and ($p\equiv 1\bmod 4$ or, $k$ is even)).
\end{minipage}
}
\right\}\eqno{(*)}
$$
Let $\mathbb{F}^{\star}_{p^k}=\langle\beta\rangle$. For prime $d^{\prime}$ dividing $d$, define 
$$S_{d^{\prime}} = \{\beta^{d^{\prime}j}: j=1,\ldots,(p^k-1)/d^{\prime}\}$$ and set
$$\displaystyle S=\bigcup_{\mbox{prime }d^{\prime}|d} S_{d^{\prime}}.$$
If $(*)$ holds, then for every $\alpha$ in $\mathbb{F}_{p^k}^{\star}\setminus S$, $x^d-\alpha$ is irreducible.

Consquently, under condition $(*)$, for a randomly chosen $\alpha$ from $\mathbb{F}_{p^k}$, the probability that $x^d-\alpha$
is irreducible is at least
$$\displaystyle 1-\sum_{\mbox{prime }d^{\prime}|d}\frac{1}{d}.$$
\end{proposition}
\begin{proof}
The argument is based on the Vahlen-Capelli criterion. 

Consider the condition ($p\equiv 1\bmod 4$ or, $k$ is even). Under this condition, $-4$ is a quadratic residue in $\mathbb{F}_{p^k}$.
This is seen as follows. If $p\equiv 1\bmod 4$, then $-1$ is a quadratic residue in $\mathbb{F}_p$; also, $4$ is a QR and so
$-4$ is a QR in $\mathbb{F}_p$ and hence a QR in $\mathbb{F}_{p^k}$. If $k$ is even, then $M=(p^k-1)/(p-1)=1+p+\cdots+p^{k-1}$ is even;
$-4=\beta^{Mj}$ for some $j\in\{0,\ldots,p-1\}$ and so $-4$ is a QR in $\mathbb{F}_{p^k}$. 

So, under the condition ($p\equiv 1\bmod 4$ or, $k$ is even), if $\alpha$ is a QNR in $\mathbb{F}_{p^k}$, then $-4\alpha$ is also
a QNR in $\mathbb{F}_{p^k}$. Consequently, $-4\alpha$ cannot have a fourth root in $\mathbb{F}_{p^k}$. 

Suppose that $4|d$ and ($p\equiv 1\bmod 4$ or, $k$ is even). If $\alpha$ is a QNR in $\mathbb{F}_{p^k}$, then $-4\alpha$ does not
have a fourth root in $\mathbb{F}_{p^k}$. So, under the condition $4|d$ and ($p\equiv 1\bmod 4$ or, $k$ is even), to show that
$x^d-\alpha$ is irreducible, it is enough to consider only the first condition of the Vahlen-Capelli criterion. On the other hand,
if $4$ does not divide $d$, then the second condition of the Vahlen-Capelli criterion does not apply. As a result, under condition
$(*)$, it is sufficient to consider the first condition of the Vahlen-Capelli criterion.

Suppose that every prime divisor $d^{\prime}$ of $d$ divides $p^k-1$ and let $\alpha\in\mathbb{F}_{p^k}^{\star}\setminus S$.
Then $\alpha$ does not have a $d^{\prime}$-th root in $\mathbb{F}_{p^k}$ for any prime divisor $d^{\prime}$ of $d$. To see this,
suppose that if possible, $\alpha$ has a $d_1$-th root $\gamma$ for some prime divisor $d_1$ of $d$. By the given condition,
$d_1|p^k-1$. Since $\alpha$ is
in $\mathbb{F}_{p^k}^{\star}\setminus S$, we can write $\alpha=\beta^i$ where $i$ is not divisible by any prime divisor $d^{\prime}$
of $d$. By assumption, $\alpha$ has a $d_1$-th root and so we can write $\alpha=\gamma^{d_1}=\beta^{jd_1}$ where $\gamma=\beta^j$
for some $j\in\{0,\ldots,p^k-2\}$. Then $\beta^i=\beta^{jd_1}$ and so $i\equiv jd_1\bmod (p^k-1)$. Since $d_1|p^k-1$, it follows 
that $d_1$ must also divide $i$ which contradicts the assumption that $i$ is not a multiple of any prime divisor of $d$.

By the above argument, it follows that under condition $(*)$, $x^d-\alpha$ is irreducible for every 
$\alpha\in\mathbb{F}_{p^k}^{\star}\setminus S$. This proves the first part of the result.

The cardinality of $S_{d^{\prime}}$ is clearly $(p^k-1)/d^{\prime}$ and so the probability that a random $\alpha$ from 
$\mathbb{F}_{p^k}$ is in $S_{d^{\prime}}$ is $1/d^{\prime}$. The exact probability that under condition $(*)$, $x^d-\alpha$
is irreducible for a random $\alpha$ is $1-\#S/(p^k-1)$. The expression for $\#S$ is obtained using inclusion-exclusion
and the stated lower bound is obtained by simply applying the union rule for probabilities.
\end{proof}

The irreducibility of $x^d-\alpha$ for prime $d$ is stated in the following result.
\begin{proposition}\label{prop-d}
Let $p$ be a prime, $k\geq 1$ be an integer, $d$ is a prime such that $d$ divides $p^k-1$. Let $\alpha$ be a random element 
of $\mathbb{F}_{p^k}$. Then the probability that $x^d-\alpha$ is irreducible over $\mathbb{F}_{p^k}$ is $1-1/d$.
\end{proposition}
\begin{proof}
Since $d$ is prime its only prime divisor is $d$ itself and by the given condition $d$ divides $p^k-1$. Also, $4$ does not
divide $d$. Hence, by Proposition~\ref{prop-suff}, the probability that $x^d-\alpha$ is irreducible is $1-1/d$.
\end{proof}

The special case of $d=2$ is of interest.
\begin{proposition}\label{prop-d=2}
Let $p$ be a prime and $k\geq 1$ be an integer. Let $\alpha$ be a random element of $\mathbb{F}_{p^k}$. Then the probability that 
$x^2-\alpha$ is irreducible over $\mathbb{F}_{p^k}$ is $1/2$.
\end{proposition}
\begin{proof}
Here $d=2$ and its only prime divisor is $2$ itself. Since $p$ is odd, $p^k-1$ is even and so $2$ divides $p^k-1$. Also, $4$ does not
divide $d$. Hence, by Proposition~\ref{prop-suff}, the probability that $x^d-\alpha$ is irreducible is $1/2$.
\end{proof}
Further applying Proposition~\ref{prop-suff}, we can deduce the following facts.
\begin{enumerate}
\item Suppose $d=4$, $p$ is an odd prime and (either $p\equiv 1\bmod 4$ or $k$ is even). Then for a random element $\alpha$ of 
$\mathbb{F}_{p^k}$, the polynomial $x^d-\alpha$ is irreducible with probability $1/2$. 
\item Suppose $d=6$ and $p$ is an odd prime such that $3$ divides $p^k-1$. Then for a random element $\alpha$ of 
$\mathbb{F}_{p^k}$, the polynomial $x^d-\alpha$ is irreducible with probability at least $1/6$. 
\end{enumerate}

\end{document}